\documentclass[10pt,a4paper]{article}

\usepackage{times}
\usepackage{enumerate}

\usepackage[english]{babel}

\usepackage{verbatim}

\usepackage{color}

\usepackage{a4wide}

\usepackage{amsfonts}
\usepackage{hyperref}

\usepackage{amsmath}

\usepackage{amscd}
\usepackage{tabularx}
\usepackage{arydshln}
\usepackage{verbatim}
\usepackage{color}
\usepackage[all]{xy}
\usepackage{mathtools} 
\usepackage{ textcomp }

\newtheorem{theorem}{Theorem}[section]

\newtheorem{proposition}[theorem]{Proposition}
\newtheorem{corollary}[theorem]{Corollary}
\newtheorem{Counter-example}[theorem]{Counter-example}
\newtheorem{remark}[theorem]{Remark}

\newtheorem{definition}[theorem]{Definition}
\newenvironment{proof}{\noindent{\it Proof.}}{\hfill$\blacksquare$}

\usepackage{amsmath}

\usepackage{tikz-cd}
\tikzset{commutative diagrams/.cd}
\usetikzlibrary{intersections}

\usepackage{amssymb}

\begin{document}

\def\Q{\mathbb Q}
\def\R{\mathbb R}
\def\N{\mathbb N}
\def\Z{\mathbb Z}
\def\C{\mathbb C}
\def\S{\mathbb S}
\def\L{\mathbb L}
\def\H{\mathbb H}
\def\K{\mathbb K}
\def\X{\mathbb X}
\def\Y{\mathbb Y}
\def\Z{\mathbb Z}
\def\E{\mathbb E}
\def\J{\mathbb J}
\def\I{\mathbb I}
\def\T{\mathbb T}
\def\H{\mathbb H}

\title{A note on tractor bundles and codimension  
  \\  two spacelike immersions}

\author{
Rodrigo Mor\'on
\\ 
 \texttt{Departamento de Matemáticas, Universidad de León} \\ 
 \texttt{$rmors@unileon.es$}\\
 \texttt{https://orcid.org/0000-0001-7651-5965}}

\date{}

\begingroup
\renewcommand{\thefootnote}{} 
\addtocounter{footnote}{-1} 
\endgroup
\begingroup
\renewcommand{\thefootnote}{} 
\addtocounter{footnote}{-1} 
\endgroup

\maketitle

\begin{abstract}
\noindent We study conformal tractor bundles from an extrinsic viewpoint, relating them to codimension two spacelike immersions into Lorentzian manifolds. We show that, at least locally, every Riemannian conformal structure admits a natural realization of its normal conformal tractor bundle as the pullback of the tangent bundle of a suitably constructed Lorentzian ambient space. Finally, we reformulate the classical equations characterizing parallel sections of the normal conformal tractor bundle in this extrinsic setting, showing that they can be expressed entirely in terms of the geometry of the associated spacelike immersion. This extrinsic perspective provides additional geometric insight into parallel standard tractors and conformal holonomy.

\end{abstract}

\hyphenation{Lo-rent-zi-an}

\section{Introduction}

Conformal geometry admits several equivalent formulations, each of them highlighting different aspects of the underlying structure and providing access to distinct classes of invariants. At its most basic level, a Riemannian conformal structure on a smooth manifold is given by an equivalence class of Riemannian metrics, where two metrics are equivalent if they differ by a factor that is a smooth positive function on the manifold. Equivalently, a conformal structure can be described in terms of vector bundle data that encodes the same geometric information. This description is provided by the theory of conformal tractor bundles, whose origins go back to the work of Thomas \cite{{Thomas}} and which has since become a central tool in conformal geometry. In this approach, a conformal structure on an $(n\geq 3)$-dimensional manifold is encoded by a rank $n+2$ vector bundle endowed with a Lorentzian bundle metric, a distinguished lightlike line subbundle, and a compatible linear connection known as the normal tractor connection, see Definition \ref{12012025utygh} and the subsequent discussion. An alternative but equivalent formulation of conformal geometry is furnished by Cartan geometries of parabolic type, see for instance \cite{Cap1}, \cite{Cap} and \cite{CS09}. Among the various equivalent viewpoints, we shall restrict ourselves in this work to the tractor formalism.

In the recent work \cite{MP03192}, conformal tractor bundles have been studied from an extrinsic viewpoint by relating them to codimension two spacelike immersions in Lorentzian manifolds. It was shown that, given a spacelike immersion 
$\Psi\colon (M^n,g)\to(\widetilde{M}^{n+2},\widetilde{g})$ together with a choice of a lightlike normal vector field $\xi$ satisfying suitable geometric conditions, and assuming the intrinsic–extrinsic compatibility Condition $(\ref{cond4})$, the pullback bundle $\Psi^*(T\widetilde{M})\to M$, endowed with the induced metric $\widetilde g$ and connection $\widetilde\nabla$, naturally acquires the structure of the normal conformal tractor bundle, \cite[Cor. 3.6]{MP03192} or Proposition \ref{261121A}. This extrinsic realization of tractor geometry is particularly appealing, since it allows one to reinterpret conformal problems in terms of submanifold theory. For example, it is well-known that the existence of parallel sections of the normal conformal tractor bundle plays a central role in the study of Einstein metrics and conformal holonomy (see \cite{G20262}, \cite{G2026} and references therein), and within this extrinsic framework, the equations characterizing the existence of such parallel sections can be translated into geometric conditions on the immersion, cf. Section \ref{22012025A}. 

However, the approach developed in \cite{MP03192} relied on the \emph{a priori} existence of codimension two spacelike immersions with prescribed properties, an assumption whose validity was not addressed there. The main purpose of the present paper is to remove this limitation and to show that the extrinsic picture described above is, at least locally, always available. More precisely, let $(M,c)$ be an $n$-dimensional Riemannian conformal structure. Fix a representative metric $g\in c$ and an admissible one-parameter family of $(1,1)$-tensors $\gamma$ on $M$ (see Section \ref{22012029687} for the definition of admissibility). Using these data, we construct a Lorentzian manifold 
$$
\widetilde{M} := \R_{>0}\times(-\delta,\delta)\times M
$$ 
with metric
$$
\widetilde{g} = d(rt)\otimes dt + dt \otimes d(rt) + t^2 \langle -,- \rangle_r^g, \qquad \langle V,W\rangle_r^g := g(\gamma(r)(V),W),
$$
where $(t,r)\in \R_{>0}\times(-\delta,\delta)$ and $V,W\in \mathfrak{X}(M)$. Note that the ambient metric $\widetilde{g}$ depends explicitly on $g$ and the family $\gamma$, and defines a Lorentzian ambient space in which $M$ can be realized as a spacelike submanifold. This family of Lorentzian manifolds was first introduced in \cite{MP03193}; a summary of their main properties is presented in Section \ref{22012029687}. For any smooth function $u\in C^\infty(M)$, we define the codimension two spacelike immersion
$$
\Psi^u \colon M \longrightarrow (\widetilde{M},\widetilde{g}), \qquad x\mapsto (e^{u(x)},0,x),
$$
whose induced metric satisfies $\Psi^{u^{*}}(\widetilde{g})=e^{2u}g.$ In other words, while $\widetilde{g}$ is fixed by $g$ and $\gamma$, the immersion $\Psi^u$ depends additionally on $u$, which scales the induced metric conformally. These immersions were also introduced in \cite{MP03193}, although with a different purpose than in the present work; their main characteristics are summarized in Section \ref{1234saderer}.

The main result, Theorem \ref{1501202439578}, shows that the immersions $\Psi^u$ recover the normal conformal tractor bundle, in the sense introduced above, when the Ricci tensor of $\widetilde{g}$ vanishes along the hypersurface $\{r=0\}\subset \widetilde{M}$. As an immediate consequence, Corollary \ref{1502948766} shows that  this condition determines the derivative of $\gamma$ at $r=0$ as follows: 
$$
g(\dot{\gamma}(0)(-),-) = 2P^g,
$$
where $P^g$ denotes the Schouten tensor of $g$, defined by 
$$
P^g := \frac{1}{n-2}\Big(\mathrm{Ric}^g - \frac{\mathrm{S}^g}{2(n-1)} g\Big),
$$
with $\mathrm{Ric}^g$ and $\mathrm{S}^g$ denoting the Ricci and scalar curvatures of the metric $g$, respectively. In this sense, the Schouten tensor naturally appears as the symmetric $(0,2)$-tensor governing the extrinsic realization of the normal conformal tractor bundle via $\Psi^u$, independently of the choice of the function $u$. Finally, Remark \ref{18012025A} shows that this condition on $\gamma$ can always be achieved, at least locally.

To conclude this work, in Section \ref{22012025A} we revisit the problem of the existence of parallel sections of the normal conformal tractor bundle. Using the extrinsic realization developed throughout this paper, we reformulate the equations characterizing parallel sections as a system expressed entirely in terms of the geometry of the spacelike immersion, Proposition \ref{230212029375}. We then show that this system coincides precisely with the classical equations for parallel sections written in a metric splitting, Proposition \ref{283475645893}. This establishes the complete equivalence between the intrinsic and extrinsic descriptions. Moreover, the extrinsic viewpoint offers additional geometric insight into the structure of parallel standard tractors and further highlights the role of codimension two immersions spacelike as a natural and effective geometric framework for the study of conformal geometry.

\section{Preliminaries on submanifold theory}

Unless stated otherwise, we assume $n\geq 2$. A smooth immersion  $\Psi:M\rightarrow (\widetilde{M}, \widetilde{g})$ of a (connected) $n$-dimensional  manifold $M$ in an $(n+2)$-dimensional Lorentzian manifold $\widetilde{M}$ is said to be spacelike if the induced metric $g:= \Psi^{*}(\widetilde{g})$ is Riemannian. Here, $\Psi^{*}(\widetilde{g})$ denotes the pullback of $\widetilde{g}$ by the immersion $\Psi$. Note that this Section could be written for the more general case of arbitrary codimension, but we decided to focus on the case that interests us, namely the study of codimension two spacelike immersions into Lorentzian signature ambients. Let $\overline{\mathfrak{X}}(M)$ denote the $C^{\infty}(M)-$module of vector fields along the spacelike immersion $\Psi$, that is, the space of smooth sections of the pullback bundle $\Psi^{*}(T\widetilde{M})\to M$. Every vector field $X \in \mathfrak{X}(\widetilde{M})$ provides, in a natural way,  the vector field $\left.X\right|_{\Psi}:= X \circ \Psi \in \overline{\mathfrak{X}}(M)$. The set of vector fields $\mathfrak{X}(M)$ may be seen as a $C^{\infty}(M)-$submodule of $\overline{\mathfrak{X}}(M)$ by means of 
$$
\mathfrak{X}(M) \to \overline{\mathfrak{X}}(M),\quad  V \mapsto T\Psi \cdot V,
$$
where $T\Psi$ denotes the differential map of the immersion $\Psi$ and $(T \Psi \cdot V)(x):=T_{x}\Psi \cdot V_{x}$ for all $x\in M$.  As usual, for $V\in \overline{\mathfrak{X}}(M)$, we have the decomposition
$$
V=T\Psi\cdot V^{\top}+V^{\bot},
$$
where $V^{\top}_{x}\in T_{x}M$ and $V^{\bot}_{x}\in (T_{x}\Psi\cdot T_{x}M)^{\bot}$ for all $x\in M$. 
We call $V^{\top}$ the tangent part of $V$ and $V^{\bot}$ the normal part of $V$. The $C^{\infty}(M)-$submodule of $\overline{\mathfrak{X}}(M)$ of all normal vector fields along $\Psi$ is denoted by $\mathfrak{X}^{\perp}(M)$. 
We explicitly write the immersions and differential maps when necessary. 

We denote by $\widetilde{\nabla}$ and $\nabla$ to the Levi-Civita connections of $\widetilde{g}$ and $g$, respectively. As usual, we also denote by $\widetilde{\nabla}$ the induced connection on $M$. 
The decomposition of the induced connection $\widetilde{\nabla}$, into tangent and normal parts, leads to the Gauss and Weingarten formulas of $\Psi$ as follows (see \cite[Chap. 4]{One83})
\begin{equation}\label{shape}
	\widetilde{\nabla}_V (T\Psi\cdot W)=T\Psi\cdot\nabla_V W + \mathrm{II}(V,W) \quad \quad \mathrm{and} \quad \quad \widetilde{\nabla}_V\zeta=- T\Psi\cdot A_{\zeta}V+\nabla^{\perp}_V\zeta,
\end{equation}
for every tangent vector fields $V,W\in\mathfrak{X}(M)$ and $\zeta\in\mathfrak{X}^{\perp}(M)$. Here, $\nabla^{\perp}$ denotes the normal connection on $M$, $\mathrm{II}$ the second fundamental form and $A_{\zeta}$ the Weingarten endomorphism (or shape operator) associated to $\zeta$. Every Weingarten endomorphism $A_{\zeta}$ is self-adjoint and the second fundamental form is symmetric. They are also related by the following formula
\begin{equation}\label{230321C}
	g\left(A_{\zeta}V,W\right)  = \widetilde{g} \left(\mathrm{II}(V,W), \zeta \right).
\end{equation}
The mean curvature vector field  is defined by
$\mathbf{H}=\frac{1}{n}\mathrm{trace}_{g}\mathrm{II}$ where $\mathrm{trace}_{g}$ denotes the trace with respect to the metric $g$.  

Now assume that there exists a global lightlike normal frame $\{\xi,\ell\}$ along $\Psi$. That is, $\xi$ and $\ell$ are two globally defined normal vector fields  along $\Psi$ satisfying $\widetilde g(\xi,\xi)=\widetilde g(\ell,\ell)=0$ and the normalization condition $\widetilde{g}(\xi,\ell)=-1$. In this case, the pullback bundle $\Psi^{*}(T\widetilde{M}) \to M$ admits a canonical vector bundle isomorphism
\begin{equation}\label{12012025gfdgfd}
\Psi^{*}(T\widetilde{M}) \simeq \underline{\mathbb{R}} \oplus TM \oplus \underline{\mathbb{R}}.
\end{equation}
Under this identification, the first trivial line bundle corresponds to $\mathrm{Span}(\xi)$, the middle summand is identified with the tangent bundle $TM$ via $T\Psi$, and the last trivial line bundle corresponds to $\mathrm{Span}(\ell)$. On the other hand, using Equation $(\ref{230321C})$, the second fundamental form of the immersion can be expressed, for all $V,W\in\mathfrak{X}(M)$, as
\begin{equation}\label{270221B}
	\mathrm{II}(V,W)=-g(A_{\ell}V,W)\,\xi-g(A_{\xi}V,W)\,\ell.
\end{equation}
Taking traces in this expression, we obtain for the mean curvature vector field that
\begin{equation}\label{H}
	\mathbf{H}=-\dfrac{1}{n}\bigg(\mathrm{trace}\,(A_{\ell})\xi+\mathrm{trace}\,(A_{\xi})\ell\bigg).
\end{equation}

\section{Conformal Tractor bundles and codimension two spacelike immersions}
This Section recalls the notions of the conformal tractor bundle and the tractor connection, following the definitions in \cite{Cap}. It also reviews the extrinsic viewpoint provided by spacelike immersions in Lorentzian geometry, as previously discussed in \cite{MP03192}.
\begin{definition}\label{12012025utygh} 
 Let $M$ be an $n$-dimensional manifold. Then:
\begin{itemize} 
	\item  A (Riemannian) conformal tractor bundle on a manifold $M$ is a rank $n+2$ real vector bundle $\mathcal{T}\to M$ endowed with a bundle metric $\mathbf{h}$ of Lorentzian signature and with a distinguished oriented lightlike line subbundle $\mathcal{T}^{1}\subset \mathcal{T}$.

	\item  A tractor connection $\nabla^{\mathcal{T}}$ on a conformal tractor bundle $\mathcal{T}\to M$ is a linear connection such that 
	$\nabla^{\mathcal{T}} \mathbf{h}=0$ and the following map $\beta$ is an isomorphism of vector bundles over $M$:
	$$
	\begin{tikzpicture}
		\node (A) at (0,0) {$TM$};
		\node (B) at (4,0) {$\mathrm{Hom}\big(\mathcal{T}^{1},(\mathcal{T}^{1})^{\perp}/\mathcal{T}^{1}\big)$};
		\node (C) at (2,-1) {$M$};
		\node (D) at  (1.3,0.25) {$\beta$};
		\draw[->] (A) -- (B);
		\draw[->] (B) -- (C);
		\draw[<-] (C) -- (A);
	\end{tikzpicture}
	$$
	where $\beta$ is given by
	\begin{equation}\label{16082024}
		\beta(V_{x})\big(\xi\big)=\nabla^{\mathcal{T}}_{V_{x}}\sigma +\,\mathcal{T}^{1}_{x},
	\end{equation}
	with $x\in M$, $V_{x}\in T_{x}M$, $\xi\in \mathcal{T}^{1}_{x}$ and  $\sigma\in \Gamma(\mathcal{T}^{1})$ any section satisfying $\sigma(x)=\xi$. 
	\end{itemize}	   
\end{definition}
Let $(\mathcal{T},\mathcal{T}^1,\mathbf{h},\nabla^{\mathcal{T}})$ be a conformal tractor bundle on a manifold $M$ endowed with a tractor connection. From $(\ref{16082024})$, it follows that every nonvanishing local section $\sigma\in\Gamma(\mathcal{T}^{1})$ provides the vector bundle isomorphism
\begin{equation}\label{12012025}
	\beta_{\sigma}\colon TM\to \Big((\mathcal{T}^{1})^{\perp} /\mathcal{T}^{1}\Big) ,\quad V \mapsto \nabla^{\mathcal{T}}_{V}\sigma+ \mathcal{T}^{1}.
\end{equation}
Thus, every nonvanishing local section $\sigma \in \Gamma(\mathcal{T}^{1})$ produces a Riemannian metric $\mathbf{h}^{\sigma}$ on $M$ by means of the formula
\begin{equation}\label{12012026fdhfg}
	\mathbf{h}^{\sigma}(V,W):=\mathbf{h}(\beta_{\sigma}(V), \beta_{\sigma}(W)),
\end{equation}
for $V, W \in \mathfrak{X}(M)$. 
Any other nonvanishing local section $\bar{\sigma}\in\Gamma(\mathcal{T}^1)$ can be written as $\bar{\sigma}=f \sigma$ for some nonvanishing smooth function $f$ on $M$. Therefore, different choices of the section $\sigma$ induce conformally related metrics on $M$ and then, a conformal class $c$ on $M$. If we start with a Riemannian conformal structure $(M,c)$ and the induced conformal structure on $M$ by means of $(\mathcal{T},\mathcal{T}^{1}, \mathbf{h}, \nabla^{\mathcal{T}})$ agrees with $c$, we say that $(\mathcal{T}, \mathcal{T}^{1}, \mathbf{h}, \nabla^{\mathcal{T}})$ is a standard conformal tractor bundle for the fixed Riemannian conformal structure. Furthermore, there are certain normalization conditions that ensure the uniqueness of the tractor connection for $n\geq 3$; for more details see \cite[\text{Sec. 2.2}]{Cap}. The unique tractor connection satisfying these normalization conditions is called the normal tractor connection. The resulting structure $(\mathcal{T}, \mathcal{T}^{1}, \mathbf{h}, \nabla^{\mathcal{T}})$ is uniquely determined by the underlying conformal structure, up to isomorphism, and is referred to as the normal conformal tractor bundle.

\begin{remark}\label{12012025ythbggfd}
{\rm Consider the Riemannian conformal structure $c$ induced on $M$ from the data $(\mathcal{T},\mathcal{T}^{1}, \mathbf{h}, \nabla^{\mathcal{T}})$. Every choice of a metric $g\in c$ provides us with a decomposition of $\mathcal{T}$ as follows.
	The metric $g\in c$ is determined by an oriented section $\sigma \in \Gamma(\mathcal{T}^{1})$ by the condition $g=\mathbf{h}^{\sigma}$, where $\mathbf{h}^{\sigma}$ is given by $(\ref{12012026fdhfg})$. 
	Then, we obtain the $g$-decomposition
	$$
	\mathcal{T} \overset{g}{\simeq } \underline{\R}\oplus TM\oplus\underline{\R},
	$$
	where $\underline{\R}$ denotes the trivial bundle $M\times\R\rightarrow M$. The first trivial bundle $M\times\R\rightarrow M$ arises from the trivialization of $\mathcal{T}^{1}$ deduced from $\sigma$. The copy of $TM$ is given by means of 
	$$
	F\colon TM\to \mathcal{T} ,\quad V_{x} \mapsto \nabla^{\mathcal{T}}_{V_{x}}\sigma.
	$$
	for every $x\in M$. The second trivial bundle $M\times\R\rightarrow M$ comes from the unique lightlike section $\delta \in \Gamma(\mathcal{T})$ such that
	$\mathbf{h}(\sigma, \delta)=-1$ and $\mathbf{h}(\delta, F(TM))=0$.} 
\end{remark}
We now turn to the setting of spacelike codimension two immersions into Lorentzian manifolds. In this context, \cite{MP03192} addresses the problem of determining when such an immersion allows one to recover the normal conformal tractor bundle as a pullback bundle naturally associated with the immersion. Although the results therein are formulated in a more general framework, in this paper we restrict our attention to the following situation.

Let $M$ be an $n$-dimensional manifold and $(\widetilde{M},\widetilde{g})$ an $(n+2)$-dimensional Lorentzian manifold. Consider a spacelike immersion $\Psi \colon M \to (\widetilde{M},\widetilde{g})$ with induced metric $g$. Assume that along $\Psi$ there exists a global lightlike normal frame $\{\xi,\ell\}$ such that the Weingarten endomorphism associated with $\xi$ satisfies $A_\xi = -\mathrm{Id},$ and that $\xi$ is parallel with respect to the normal connection, that is, $\nabla^{\perp}\xi = 0.$ Under these assumptions, there is a natural choice of standard conformal tractor bundle for the Riemannian conformal structure $(M,[g])$, as described in \cite[Prop. 3.1]{MP03192}. It is given by the quadruple
$$
\bigl(\Psi^{*}(T\widetilde{M}), \mathrm{Span}(\xi), \widetilde{g}, \widetilde{\nabla}\bigr),
$$
where $\Psi^{*}(T\widetilde{M}) \to M$ denotes the pullback of the tangent bundle of $\widetilde{M}$, the distinguished lightlike line subbundle $\mathcal{T}^{1}$ is given by $\mathrm{Span}(\xi)$, the tractor metric $\mathbf{h}$ is the Lorentzian bundle metric induced by $\widetilde{g}$, and the tractor connection is the induced connection $\widetilde{\nabla}$. We refer to $\bigl(\Psi^{*}(T\widetilde{M}), \mathrm{Span}(\xi), \widetilde{g}, \widetilde{\nabla}\bigr)$ as the tractor bundle associated with the pair $(\Psi, \xi)$.
\begin{remark}\label{20012025}
{\rm In the above setting, it is worth noting that the map defined in $(\ref{12012025})$ reduces to $\beta_{\xi}(V)= -\,T\Psi\cdot A_{\xi}(V) + \mathrm{Span}(\xi)= T\Psi\cdot V + \mathrm{Span}(\xi),$ and hence, by $(\ref{12012026fdhfg})$, it follows that $g=\mathbf{h}^{\xi}$. Taking into account Remark~\ref{12012025ythbggfd} and the Weingarten formula $(\ref{shape})$, since $\xi$ is parallel with respect to the normal connection, then the decomposition given in $(\ref{12012025gfdgfd})$ coincides with the $g$-decomposition of the tractor bundle. This reflects the fact that, within the present extrinsic construction, the distinguished lightlike line subbundle $\mathcal{T}^{1}$ is identified with $\mathrm{Span}(\xi)$, so that a global section of $\mathcal{T}^{1}$ is fixed. Such a choice is stronger than the data of a general conformal tractor bundle and is equivalent to fixing a representative metric $g$ in the conformal class $[g]$. Consequently, in the situation described above, this extrinsic viewpoint naturally yields the $g$-decomposition of the tractor bundle.
}
\end{remark}

We conclude this Section by recalling a characterization of when the induced connection coincides with the normal tractor connection. More precisely, we state a particular case of \cite[Cor. 3.6]{MP03192} adapted to the present setting.

\begin{proposition}\label{261121A}
Let $\bigl(\Psi^{*}(T\widetilde{M}), \mathrm{Span}(\xi), \widetilde{g}, \widetilde{\nabla}\bigr)$ be the tractor bundle associated with the pair $(\Psi, \xi)$. Then $\bigl(\Psi^{*}(T\widetilde{M}), \mathrm{Span}(\xi), \widetilde{g}, \widetilde{\nabla}\bigr)$ is the normal conformal tractor bundle if and only if the Ricci tensor of $g$ satisfies
\begin{equation}\label{cond4}
\mathrm{Ric}^{g}(V, W)
=\frac{n}{2}\|\mathbf{H}\|^{2} g(V,W)
-(n-2)\,\widetilde{g}(\mathbf{H}, \xi)\, g\bigl(V, A_{\ell} W\bigr),
\end{equation}
for every $V,W\in \mathfrak{X}(M)$.
\end{proposition}
\begin{remark}\label{04022026}
{\rm Although the condition given in Proposition \ref{261121A} makes sense for two dimensional Riemannian conformal structures, in this case it does not provide uniqueness for the tractor connection.}
\end{remark}

\section{A family of Lorentzian manifolds associated with a conformal structure}\label{22012029687}

We now introduce a family of Lorentzian manifolds that will serve as target spaces for the class of immersions studied in this work. Since this family was previously introduced in \cite{MP03193}, we restrict ourselves to recalling the aspects that are essential for the developments that follow.

Let $(M,c)$ be an $n$-dimensional Riemannian conformal structure. A smooth $1$-parameter family $\gamma\colon \R \to  \mathcal{T}_{(1,1)}M$ is said to be  admissible if it satisfies the following conditions:
	\begin{enumerate}
		\item For every $r\in\mathbb{R}$, the tensor field $\gamma(r)$ is self-adjoint  with respect to any metric $g\in c$.
		\item   $\gamma(0)=\mathrm{Id}$.
		\item There exists $\delta >0$ such that $\gamma(r)$ is not singular for all $|r|< \delta$.
	\end{enumerate}
\noindent The smoothness of $\gamma$ is understood in the sense that for every vector field $V\in \mathfrak{X}(M)$ and every point $x\in M$, the derivative
$$
\dot{\gamma}(r)(V_{x})=\lim_{\varepsilon\to 0}\frac{\gamma(r+\varepsilon)(V_{x})- \gamma(r)(V_{x})}{\varepsilon}\in T_{x}M
$$
exists. In particular, the tensor $\dot{\gamma}(0)\in \mathcal{T}_{(1,1)}M$ is well-defined. Let us note that the condition $3$ in the above definition can be deleted when $M$ is compact and, at least locally, $\delta$ always exists in the general case.

Let us fix a metric $g\in c$ and an admissible family $\gamma$. For each $r \in \R$, we define a symmetric tensor on $M$ by
\[
\langle V,W\rangle_{r}^{g}=g\Big(\gamma(r)(V), W\Big).
\]
Clearly, $\langle \,,\,\rangle_{0}^{g}=g$,  so this construction yields a smooth $1$-parameter deformation of the metric $g$. Moreover, by admissibility, the tensors
$\langle \,,\,\rangle_{r}^{g}$ are positive definite for $|r|< \delta$. We now consider the manifold 
\begin{equation}\label{14012025iruf}
\widetilde{M}:= B\times M, \quad B:=\R_{>0}\times (-\delta, +\delta),
\end{equation}
with coordinates $(t, r)$ on $B$, endowed with the Lorentzian metric
\begin{equation}\label{14012025gfjjgf}
	\widetilde{g}= d(r t)\otimes dt+ dt \otimes d(r t)+ t^2\langle -,- \rangle_{r}^{g}.
\end{equation}
We denote by $ \mathfrak{L}(M)$ the space of vector fields on $\mathfrak{X}(\widetilde{M})$ obtained as natural lifts of vector fields on $M$. By a slight abuse of notation, we use the same notation for a vector field $V \in \mathfrak{X}(M)$ and for its lift to $\mathfrak{X}(\widetilde{M})$. 

The Levi-Civita connection $\widetilde{\nabla}$ of $(\widetilde{M}, \widetilde{g})$ satisfies
	$$
		\widetilde{\nabla}_{\partial_{t}} \partial_{t}=\widetilde{\nabla}_{\partial_{r}} \partial_{r}=0, \quad \quad \widetilde{\nabla}_{\partial_{t}}\partial_{ r}=\widetilde{\nabla}_{\partial_{r}}\partial_{t}=\dfrac{1}{t}\partial_{r},
	$$
	$$	
		\widetilde{\nabla}_{V}\partial_{t}=\dfrac{1}{t}V,\quad \widetilde{\nabla}_{V}\partial_{r}=\frac{1}{2}\gamma(r)^{-1}(\dot{\gamma}(r)(V)),
	$$	
	$$
		\widetilde{\nabla}_{V}W|_{r=0}=-\frac{1}{2t}\widetilde{g}(\dot{\gamma}(0)(V), W)\partial_{t}-\frac{1}{t^2}\widetilde{g}(V,W)\partial_{r}+ \nabla_{V}W,
	$$	
	where $V, W\in \mathfrak{L}(M)$ and $\nabla$ denotes the Levi-Civita connection of  $(M,g)$, \cite[Prop. 3.6]{MP03193}. Moreover, the Ricci tensor $\widetilde{\mathrm{Ric}}$ of  $(\widetilde{M}, \widetilde{g})$ restricted to $r=0$ satisfies
$$
	\widetilde{\mathrm{Ric}}|_{r=0}(\partial_{t},\partial_{t})=\widetilde{\mathrm{Ric}}|_{r=0}(\partial_{t},V)=0, \quad V\in\mathcal{L}(M),
$$
and
\begin{equation}\label{220721A}
		\widetilde{\mathrm{Ric}}|_{r=0}(V,W)=\mathrm{Ric}^g(V,W)-\dfrac{\mathrm{trace}(\dot{\gamma}(0))}{2}g(V,W)-\left(\dfrac{n-2}{2}\right)g\left(\dot{\gamma}(0)(V),W\right),
\end{equation}
where $V, W\in \mathfrak{L}(M)$ and $\mathrm{Ric}^g$ denotes the Ricci tensor of $(M,g)$, \cite[Cor. 3.9]{MP03193}.  As a consequence, the vanishing of the Ricci tensor along $r=0$ can be characterized as follows:
	\begin{itemize}
		\item For $n=2$, it satisfies that
		$\widetilde{\mathrm{Ric}}|_{r=0}(V,W)= 0$ if and only if $$\mathrm{trace}(\dot{\gamma}(0))=2K^g,$$  where $K^g$ is the Gauss curvature of  $g$.
		\item For $n\geq 3$, it satisfies that
		$\widetilde{\mathrm{Ric}}|_{r=0}(V,W)= 0$  if and only if $$g(\dot{\gamma}(0)(-), -)=2P^g,$$  where $P^g$ is the Schouten tensor of $g$.
	\end{itemize}
\begin{remark}
	{\rm As pointed out in \cite[Prop. 3.3]{MP03193}, this family of Lorentzian manifolds gives rise to pre-ambient spaces for $(M,c)$ in the sense introduced in \cite{Cap}. Although this perspective will not be required in what follows, we consider it worth mentioning in view of the inherent difficulty of constructing such spaces.}
\end{remark}

\section{Realization of the normal conformal tractor bundle via a class of codimension two spacelike immersions}\label{1234saderer}

Let $(M,c)$ be an $n$-dimensional Riemannian conformal structure. Throughout this Section, we fix a metric $g\in c$ and an admissible family $\gamma$. Associated with these data, we consider the Lorentzian manifold $(\widetilde{M},\widetilde{g})$ constructed in \eqref{14012025iruf} and $(\ref{14012025gfjjgf})$.

For any function $u \in C^\infty(M)$, we define the map
\begin{equation}\label{immer}
	\Psi^{u}\colon M \to (\widetilde{M},\widetilde{g}), \qquad 
	x \mapsto \big(e^{u(x)},0,x\big),
\end{equation}
which yields a spacelike immersion satisfying $\Psi^{u^{*}}(\widetilde{g})=e^{2u}g.$ The differential map of $\Psi^{u}$ is given by
\begin{equation}\label{220221C}
	T\Psi^{u}(V)=e^{u}V(u)\,\partial_{t}\big|_{\Psi^{u}}+V\big|_{\Psi^{u}},
\end{equation}
for every $V\in\mathfrak{X}(M)$. A straightforward computation based on \eqref{220221C} shows that the vector fields
$$
	\xi^{u}=e^{u}\partial_{t}\big|_{\Psi^{u}}\quad \mathrm{and} \quad\ell^{u}=e^{-u}\frac{\|\nabla u\|^{2}}{2}\partial_{t}\big|_{\Psi^{u}}-e^{-2u}\partial_{r}\big|_{\Psi^{u}}+e^{-2u}\nabla u\big|_{\Psi^{u}}
$$
span the normal bundle of $\Psi^{u}$, where $\nabla u$ denotes the gradient of the function $u$ with respect to the metric $g$, and $\|\cdot\|$ denotes the norm induced by $g$. Moreover, one readily verifies that $\{\xi^{u},\ell^{u}\}$ defines a global lightlike normal frame. 

Let $A_{\xi^{u}}$ and $A_{\ell^{u}}$ denote the Weingarten endomorphisms corresponding to the lightlike normal vector fields $\xi^{u}$ and $\ell^{u}$, respectively. Then
\begin{equation}\label{140102025jfhd}
A_{\xi^{u}}=-\mathrm{Id}\quad \mathrm{and} \quad A_{\ell^{u}}=e^{-2u}\left[\frac{\dot{\gamma}(0)-\|\nabla u\|^{2}\,\mathrm{Id}}{2}+g(\nabla u,\mathrm{Id})\,\nabla u-\nabla\nabla u\right],
\end{equation}
where $\nabla\nabla u$ denotes the covariant derivative of the gradient of $u$ with respect to the Levi-Civita connection $\nabla$ of $(M,g)$, that is, $\nabla\nabla u(V):=\nabla_{V}\nabla u$ for all $V\in\mathfrak{X}(M)$, see details in \cite[Prop. 4.2]{{MP03193}}. Furthermore, both $\xi^{u}$ and $\ell^{u}$ are parallel with respect to the normal connection, \cite[Cor. 4.3]{{MP03193}}. Using Equation \eqref{270221B}, the second fundamental form $\mathrm{II}^{u}$ of the immersion $\Psi^{u}$ can be expressed as
\begin{equation}\label{II2345}
\mathrm{II}^{u}(V,W)=-g\Big(\frac{\dot{\gamma}(0)(V)-\|\nabla u\|^{2}V}{2}+V(u)\nabla u-\nabla_{V}\nabla u,W\Big)\,\xi^{u}+e^{2u}g(V,W)\,\ell^{u},
\end{equation}
for all vector fields $V,W\in\mathfrak{X}(M)$. In particular, from Equation \eqref{H} we obtain the mean curvature vector field
\begin{equation}\label{020321A}
\mathbf{H}^{u}=\frac{e^{-2u}}{n}\left(\Delta u-\frac{\mathrm{trace}(\dot{\gamma}(0))-(n-2)\|\nabla u\|^{2}}{2}\right)\xi^{u}+\ell^{u},
\end{equation}
where $\Delta$ denotes the Laplace operator associated with the metric $g$.

At this point, all the required framework is in place to state the main theorem of the paper on immersions recovering the normal conformal tractor bundle. 
\begin{theorem}\label{1501202439578}
Let $\bigl(\Psi^{u^{*}}(T\widetilde{M}), \mathrm{Span}(\xi^u), \widetilde{g}^u, \widetilde{\nabla}^u\bigr)$ be the tractor bundle associated with the pair $(\Psi^u, \xi^u)$, where $\widetilde{g}^u$ denotes the restriction of the metric $\widetilde{g}$ along the immersion $\Psi^u$, and $\widetilde{\nabla}^u$ is the induced connection along $\Psi^u$. Then $\bigl(\Psi^{u^{*}}(T\widetilde{M}), \mathrm{Span}(\xi^u), \widetilde{g}^u, \widetilde{\nabla}^u\bigr)$ is the normal conformal tractor bundle if and only if $\widetilde{\mathrm{Ric}}|_{r=0}(V,W)= 0$ for every $V, W\in \mathfrak{L}(M)$.
\end{theorem}
\begin{proof}
Proposition \ref{261121A} reduces the problem to verifying that the immersions $\Psi^u$ introduced in $(\ref{immer})$ satisfy Condition $(\ref{cond4})$ for the metric $e^{2u}g$ if and only if $\widetilde{\mathrm{Ric}}|_{r=0}(V,W)= 0$ for every $V, W\in \mathfrak{L}(M)$. Taking into account Equations $(\ref{140102025jfhd})$ and $(\ref{020321A})$, Condition $(\ref{cond4})$ can be rewritten as
\begin{align*}
   \mathrm{Ric}^{e^{2u}g}(V, W) & =\left(-\Delta u+\frac{\mathrm{trace}(\dot{\gamma}(0))}{2}-(n-2)\|\nabla u\|^2\right)g(V,W)\\
   & + \left(\dfrac{n-2}{2}\right)g\big(\dot{\gamma}(0)(V),W\big)+(n-2)g\big(\nabla u,V\big)g\big(\nabla u,W\big)-(n-2)g\big(\nabla_V\nabla u,W\big).
\end{align*}
On the other hand, it is well-known that the Ricci tensors of two conformally related metrics satisfy
$$
\mathrm{Ric}^{e^{2u}g}=\mathrm{Ric}^{g}-\Delta u\,g-(n-2)\|\nabla u\|^2 g+(n-2)g\big(\nabla u,-\big)g\big(\nabla u,-\big)-(n-2)g\big(\nabla_V\nabla u,W\big).
$$
Combining these expressions and using Equation~\eqref{220721A} completes the proof.
\end{proof}
\begin{corollary}\label{1502948766}
Let $\bigl(\Psi^{u^{*}}(T\widetilde{M}), \mathrm{Span}(\xi^u), \widetilde{g}^u, \widetilde{\nabla}^u\bigr)$ be the tractor bundle associated with the pair $(\Psi^u, \xi^u)$. Then $\bigl(\Psi^{u^{*}}(T\widetilde{M}), \mathrm{Span}(\xi^u), \widetilde{g}^u, \widetilde{\nabla}^u\bigr)$ is the normal conformal tractor bundle if and only if:
\begin{itemize}
		\item For $n=2$, it satisfies that $$\mathrm{trace}(\dot{\gamma}(0))=2K^g,$$  where $K^g$ is the Gauss curvature of  $g$.
		\item For $n\geq 3$, it satisfies that $$g(\dot{\gamma}(0)(-), -)=2P^g,$$  where $P^g$ is the Schouten tensor of $g$.
	\end{itemize}
\end{corollary}
\begin{remark}
{\rm We recall from Remark \ref{04022026} that, in the case $n=2$, the normalization condition does not constitute a genuine normalization, as it does not ensure uniqueness of the tractor connection.}
\end{remark}
\begin{remark}
{\rm As shown in Theorem \ref{1501202439578}, the necessary condition for these immersions to recover the normal tractor connection does not depend on the immersion itself, but solely on a curvature condition of the Lorentzian metric $\widetilde{g}$ along $r=0$. In the language of pre-ambient spaces, this corresponds to requiring the Ricci tensor to vanish along the bundle of scales. Consequently, this result is consistent with \cite[Sec. 2.5]{Cap}, where the same curvature condition appears, although the construction of the normal tractor connection there is of a different nature, being obtained by descending the pre-ambient Levi-Civita connection via a quotient procedure.}
\end{remark}
\begin{remark}\label{18012025A}
{\rm The normalization condition given in Theorem \ref{1501202439578}, or equivalently in Corollary \ref{1502948766}, can always be satisfied, at least locally. Indeed, for $n\geq 3$, consider the curve $\gamma(r) = \mathrm{Id} + 2r\widehat{P}^g,$ where $P^g(V,W) = g(\widehat{P}^g(V), W)$ for $V,W \in \mathfrak{X}(M)$. For any point $x \in M$, there exists an open neighbourhood $x \in \mathcal{O} \subset M$ such that $\gamma$ defines an admissible family on $\mathcal{T}_{(1,1)}\mathcal{O}$. Note that the specific choice $\gamma(r) = \mathrm{Id} + 2r\widehat{P}^g$ can be replaced by any curve satisfying $\gamma(0) = \mathrm{Id}$ and $\dot{\gamma}(0) = 2\widehat{P}^g$. This observation allows us to regard the normal conformal tractor bundle as the pullback bundle associated with a codimension two spacelike immersion.}
\end{remark}
\begin{remark} 
{\rm Assume that $\widetilde{\mathrm{Ric}}|_{r=0}= 0$. Then $\Psi^u$ satisfies Equation $(\ref{cond4})$. It follows from this equation that $(M,e^{2u}g)$ is Einstein if and only if the immersion $\Psi^u$ is totally umbilical and, in particular, the conformal class $c$ contains an Einstein metric if and only if there exists some $u\in\mathcal{C}^{\infty}(M)$ such that $\Psi^u$ is totally umbilical. To prove the first equivalence, it suffices to note from $(\ref{cond4})$ that $\operatorname{Ric}^{e^{2u}g}=he^{2u}g$ if and only if $A_{\ell^u}=\mu\mathrm{Id},$ for some smooth functions $h,\mu\in\mathcal{C}^{\infty}(M)$. Equation $(\ref{H})$ then implies that necessarily $$ \mu=\widetilde{g}(\mathbf{H}^u,\ell^u), \qquad \|\mathbf{H}^u\|^{2}=-2\,\widetilde{g}(\mathbf{H}^u,\xi^u)\,\widetilde{g}(\mathbf{H}^u,\ell^u), $$ and the equivalence follows directly from these identities. Moreover, in this case the Ricci tensor satisfies $$\mathrm{Ric}^{e^{2u}g}=(n-1)\|\mathbf{H^u}\|^{2}\,e^{2u}g$$ and, by Exercise $21$ in \cite[p. 96 ]{One83}, for $n\geq 3$ the function $\|\mathbf{H}^u\|^{2}$ is necessarily constant. This is consistent with the fact that, when $\operatorname{trace}(\dot{\gamma}(0)) = \frac{S^g}{n-1}$ (which holds when $\widetilde{\mathrm{Ric}}|_{r=0}= 0$), and using the transformation law for the scalar curvature under conformal changes, Formula $(\ref{020321A})$ reduces to $$ \mathbf{H}^{u} = -\frac{1}{2n(n-1)}\, S^{e^{2u}g}\, \xi^u + \ell^u, $$ and therefore $\|\mathbf{H}^u\|^2 = \frac{S^{e^{2u}g}}{n(n-1)}.$ It is worth noting that this last expression for $\mathbf{H}^{u}$ extends \cite[Cor. 4.5]{PPR} and \cite[Cor. 3.7]{PR13}.} 
\end{remark}

\section{Parallel sections of the normal conformal tractor bundle}\label{22012025A}

To conclude this work, this Section is devoted to a reformulation of the existence problem for parallel sections of the normal conformal tractor bundle, expressed in terms of the pullback bundle associated with a spacelike immersion from the family $(\ref{immer})$, in the case $n \geq 3$. In what follows, we fix a choice of admissible family $\gamma$ satisfying the normalization condition of Corollary \ref{1502948766}, so that the associated tractor bundle is indeed the normal conformal tractor bundle. Since the normal conformal tractor bundle is unique up to isomorphism, it follows that all normal conformal tractor bundles appearing in Theorem \ref{1501202439578} are mutually isomorphic. Therefore, without loss of generality, we may restrict our attention to the case $u = 0$.  That is, we shall work with the normal conformal tractor bundle associated with the pair $(\Psi^0, \xi^0)$, which, in order to simplify the notation, will be denoted simply by $\bigl(\Psi^{*}(T\widetilde{M}), \mathrm{Span}(\xi), \widetilde{g}, \widetilde{\nabla}\bigr),$ with the understanding that the superscript $0$ is omitted in all expressions from now on.

Since the normal tractor connection is precisely the induced connection, the study of the existence of parallel sections is equivalent to proving the existence of $W\in\overline{\mathfrak{X}}(M)$ such that
\begin{equation}\label{17012025hff}
\widetilde{\nabla}_V W=0,\, \text{ for every }V\in\mathfrak{X}(M).
\end{equation}
Observe that we can decompose $W$ as 
$$
W=T\Psi\cdot W^{\top}+W^{\bot},
$$
where $W^{\top}\in\mathfrak{X}(M)$ and $W^{\bot}\in\mathfrak{X}^{\perp}(M)$. Therefore, as a consequence of $(\ref{shape})$ we have that
$$
\widetilde{\nabla}_V W=T\Psi\cdot\nabla_V W^{\top} + \mathrm{II}(V,W^{\top})- T\Psi\cdot A_{W^{\bot}}V+\nabla^{\perp}_VW^{\bot}.
$$
Then, Equation $(\ref{17012025hff})$ admits the following equivalent formulation.
\begin{proposition}\label{230212029375}
   Let $\bigl(\Psi^{*}(T\widetilde{M}), \mathrm{Span}(\xi), \widetilde{g}, \widetilde{\nabla}\bigr)$ be the normal conformal tractor bundle associated with the pair $(\Psi,\xi).$ Then $W\in\overline{\mathfrak{X}}(M)$ is a parallel section of $\bigl(\Psi^{*}(T\widetilde{M}), \mathrm{Span}(\xi), \widetilde{g}, \widetilde{\nabla}\bigr)$ if and only if:
\begin{equation}\label{17012025utrhgf}
\left\{
\begin{aligned}
&\nabla_V W^{\top}=A_{W^{\bot}} V, \\
&\mathrm{II}(V, W^{\top}) + \nabla^{\perp}_V W^{\bot} = 0.
\end{aligned}
\right.
\end{equation}
   for every $V\in\mathfrak{X}(M)$. 
\end{proposition}
\begin{remark}
{\rm Suppose that $W^{\perp}$ is parallel with respect to $\nabla^{\perp}$. Then, by the second equation in \eqref{17012025utrhgf}, we obtain $\mathrm{II}(V, W^{\top})=0$ for every $V\in\mathfrak{X}(M)$. Hence, by Equation $(\ref{II2345})$, it follows that $W^{\top}=0$. Consequently, there are no non-trivial parallel sections with purely tangential component, that is, with $W^{\perp}=0$. While elementary, this observation exemplifies the effectiveness of the proposed framework in translating the search for parallel tractor sections into concrete geometric conditions.}
\end{remark}
As $W^{\bot}\in\mathfrak{X}^{\perp}(M)$ and $\{\xi,\ell\}$ form a global normal frame, $W^{\bot}$ admits the decomposition 
$$W^{\bot}=W_1\xi+W_2\ell,\quad W_1,W_2\in\mathcal{C}^{\infty}(M). 
$$ 
Using $(\ref{140102025jfhd})$, we obtain 
$$
A_{W^{\bot}}V=W_1A_{\xi}V+W_2A_{\ell}V=-W_1V+W_2\widehat{P}^{g}(V),
$$ 
where $\widehat{P}^{g}$ is defined in Remark \ref{18012025A}. Consequently, the first equation in $(\ref{17012025utrhgf})$ can be rewritten as 
$$
\nabla_V W^{\top}=-W_1V+W_2\widehat{P}^{g}(V). 
$$
Now, taking into account that $\xi$ and $\ell$ are parallel with respect to the normal connection, and using the formula $(\ref{II2345})$ for the second fundamental form, it is a direct computation to show that the second equation in $(\ref{17012025utrhgf})$ is equivalent to 
$$
\left(-P^g(V,W^{\top})+V(W_1)\right)\xi+\left(g(V,W^{\top})+V(W_2)\right)\ell=0,
$$
that is, to 
$$
\left\{
\begin{aligned}
&\widehat{P}^g(W^{\top})=\nabla W_1, \\
&W^{\top}=-\nabla W_2.
\end{aligned}
\right.
$$
We may therefore conclude that the problem of finding parallel sections of the normal conformal tractor bundle can be reformulated as follows.
\begin{proposition}\label{283475645893}
   Let $\bigl(\Psi^{*}(T\widetilde{M}), \mathrm{Span}(\xi), \widetilde{g}, \widetilde{\nabla}\bigr)$ be the normal conformal tractor bundle associated with the pair $(\Psi,\xi).$ Let $W\in\overline{\mathfrak{X}}(M)$ be written as $$
   W=T\Psi\cdot W^{\top}+W_1\xi+W_2\ell,
   $$
   with $W^{\top}\in\mathfrak{X}(M)$ and $W_1,W_2\in\mathcal{C}^{\infty}(M)$. Then $W$ is a parallel section of $\bigl(\Psi^{*}(T\widetilde{M}), \mathrm{Span}(\xi), \widetilde{g}, \widetilde{\nabla}\bigr)$ if and only if:
   \begin{equation}\label{sistecs} 
   \left\{ 
   \begin{aligned} 
   &\nabla_V W^{\top}=-W_1V+W_2\widehat{P}^{g}(V),\\ &\widehat{P}^g(W^{\top})=\nabla W_1, \\ 
   &W^{\top}=-\nabla W_2, \end{aligned} 
   \right. 
   \end{equation} 
   for every $V\in\mathfrak{X}(M)$. 
\end{proposition}
\begin{remark}
{\rm System $(\ref{sistecs})$ coincides exactly with the classical equations characterizing parallel sections of the normal conformal tractor bundle, written in terms of the $g$-decomposition introduced in Remarks \ref{12012025ythbggfd} and \ref{20012025}. In particular, this shows that the present formulation $(\ref{17012025utrhgf})$ is fully equivalent to the standard tractor description. For the classical approach, see \cite{G2026}.}
\end{remark}

\hspace{2cm}
\noindent{\bf Statements and Declarations}
\begin{itemize}
\item Funding: The author is  partially supported  by Spanish MICINN and ERDF project PID2024-156031NB-I00.
\item Competing Interests: The author has no relevant financial or non-financial interests to disclose.
\end{itemize}


\begin{thebibliography}{999}

\bibitem{Cap1} A. {\v C}ap and A.R. Gover, {\rm Tractor calculi for parabolic geometries}, {\it Trans. Amer. Math. Soc.}, {\bf 354} (2002), 1511--1548. 


\bibitem{Cap} A. {\v C}ap and A.R. Gover, {\rm Standard tractors and the conformal ambient metric construction}, {\it Ann. Global Anal. Geom.}, {\bf 24}(3) (2003), 231--259. 


\bibitem{CS09} A. \v{C}ap and J. Slov\'{a}k, {\it Parabolic Geometries I: Background and General Theory}, Math. Surveys Monogr. \textbf{154}, AMS, 2009. 


\bibitem{G20262} A.R. Gover, Almost Einstein and Poincaré–Einstein manifolds in Riemannian signature, {\it J. Geom. Phys.}, {\bf 60} (2010), 182--204.


\bibitem{G2026} A.R. Gover and P. Nurowski, Obstructions to conformally Einstein metrics in $n$ dimensions, {\it J. Geom. Phys.}, {\bf 56} (2006), 450--484.



\bibitem{MP03192} R. Morón and F.J. Palomo, Normal tractor conformal bundles and codimension two spacelike submanifolds in Lorentzian manifolds, {\it Differ. Geom. Appl.}, \textbf{82} (2022), Paper No. 101897, 10 pp.

\bibitem{MP03193} R. Morón and F.J. Palomo, Codimension two spacelike submanifolds in Lorentzian manifolds and conformal structures, {\it J. Math. Anal. Appl.}, \textbf{536} (2024), No. 1, Paper No. 128170.


\bibitem{One83} B. O'Neill,  {\it Semi-Riemannian Geometry with Applications to Relativity}, {\rm Academic Press}, New York, 1983. 



\bibitem{PPR} O. Palmas,   F.J. Palomo  and A. Romero, On the total mean curvature of a compact space-like submanifold in Lorentz–Minkowski spacetime, {\it Proc. Roy. Soc. Edinburgh Sect. A,} {\bf 148} (2018), 199--210. 


\bibitem{PR13} F.J. Palomo and A. Romero, On spacelike surfaces in four-dimensional Lorentz-Minkowski spacetime
through a light cone, {\it Proc. Roy. Soc. Edinburgh Sect. A,}, {\bf 143} (2013), 881--892. 


\bibitem{Thomas} T.Y. Thomas,  {\rm On conformal geometry}, {\it Proc.  Nat.  Acad.  Sci.  USA}, {\bf 12} (1926),  352--359. 

\end{thebibliography}
\end{document}